\documentclass[reqno]{amsart}
\usepackage{hkgrpd}

\newcommand{\Exterior}{\mathchoice{{\textstyle\bigwedge}}%
    {{\bigwedge}}%
    {{\textstyle\wedge}}%
    {{\scriptstyle\wedge}}}

\theoremstyle{theorem}
\newtheorem{maintheorem}{Theorem}

\newcommand{\defn}[1]{\textit{\textbf{#1}}}

\title[Hyperk\"ahler metrics on groupoids]
{
Hyperk\"ahler metrics near Lagrangian submanifolds and symplectic groupoids
}

\author{Maxence Mayrand}
\date{November 16, 2021}
\address{Department of Mathematics, University of Toronto, 
Canada}
\email{mayrand@math.toronto.edu}

\begin{document}

\begin{abstract}
The first part of this paper is a generalization of the Feix--Kaledin theorem on the existence of a hyperk\"ahler metric on a neighbourhood of the zero section of the cotangent bundle of a K\"ahler manifold.
We show that the problem of constructing a hyperk\"ahler structure on a neighbourhood of a complex Lagrangian submanifold in a holomorphic symplectic manifold reduces to the existence of certain deformations of holomorphic symplectic structures.
The Feix--Kaledin structure is recovered from the twisted cotangent bundle.
We then show that every holomorphic symplectic groupoid over a compact holomorphic Poisson surface of K\"ahler type has a hyperk\"ahler structure on a neighbourhood of its identity section.
More generally, we reduce the existence of a hyperk\"ahler structure on a symplectic realization of a holomorphic Poisson manifold of any dimension to the existence of certain deformations of holomorphic Poisson structures adapted from Hitchin's unobstructedness theorem.
\end{abstract}

\maketitle

\section{Introduction}

A hyperk\"ahler manifold, as first defined by Calabi
\cite{calabi}, is a Riemannian manifold $(M, g)$ with three complex structures $I, J, K$ that are K\"ahler with respect to $g$ and satisfy the quaternionic identities $I^2 = J^2 = K^2 = IJK = -1$. 
One of the main features of this geometry is its deep connection with symplectic geometry.
In particular, every hyperk\"ahler manifold has an underlying holomorphic symplectic structure, namely $(I, \omega_J + i \omega_K)$, where $\omega_J$ and $\omega_K$ are the K\"ahler forms of $J$ and $K$, respectively.
Conversely, Yau's solution to the Calabi conjecture \cite{yau} implies that every \emph{compact} holomorphic symplectic manifold of K\"ahler type is endowed with a hyperk\"ahler structure, as first observed by Beauville \cite[Proposition 4]{beauville}.

In the non-compact case, there is no such general existence result, and the problem of constructing hyperk\"ahler metrics on non-compact holomorphic symplectic manifolds has generated a lot of research since Calabi's first examples \cite{calabi} on $T^*\CP^n$.
Prominent examples include coadjoint orbits of complex semisimple Lie algebras \cite{kronheimer-coadjoint, kronheimer-nilpotent, biquard-nahm, kovalev-nahm}, cotangent bundles of hermitian symmetric spaces \cite{biquard-gauduchon-symmetric, biquard-gauduchon-geometrie}, the ALE spaces on resolutions of Kleinian singularities \cite{kronheimer-ale} later generalized to Nakajima quiver varieties \cite{nakajima-instantons}, and various moduli spaces in gauge theory, such as moduli spaces of instantons \cite{maciocia}, monopoles \cite{atiyah-hitchin}, Higgs bundles \cite{hitchin-self}, and Nahm's equations \cite{kronheimer-coadjoint, kronheimer-nilpotent, biquard-nahm, kovalev-nahm, kronheimer-cotangent, bielawski-actions}.
Perhaps the most general existence result is that of Feix \cite{feix-cotangent} and Kaledin \cite{kaledin-2, kaledin-1}, who independently showed that the cotangent bundle of any K\"ahler manifold has a hyperk\"ahler structure on a neighbourhood of its zero section. 

The first result of this paper is a generalization of the Feix--Kaledin theorem: we reduce the problem of constructing a hyperk\"ahler metric on a neighbourhood of a complex Lagrangian submanifold in a holomorphic symplectic manifold to the existence of certain deformations of holomorphic symplectic structures (Theorem \ref{lagrangian-theorem}). 
The Feix--Kaledin theorem is recovered from the twisted cotangent bundle (Example \ref{feix-kaledin-example}).
The proof uses twistor theory, as in Feix's proof \cite{feix-cotangent}.
We then use this result to associate a hyperk\"ahler manifold to every compact holomorphic Poisson surface endowed with a K\"ahler metric, as we now explain.

The notion of symplectic groupoids, introduced by Weinstein \cite{weinstein-groupoids}, Karasev \cite{karasev-quantization, karasev-analogues}, and Zakrzewski \cite{zakrzewski-1, zakrzewski-2}, has played a central r\^ole in Poisson geometry, especially for the problem of quantization.
Every smooth Poisson manifold $(X, \sigma)$ has an associated Lie algebroid $(T^*X)_\sigma$ on its cotangent bundle, so one can ask for a Lie groupoid $\G \rightrightarrows X$ integrating it.
If it exists, $\G$ is a \emph{symplectic groupoid}, i.e.\ it is endowed with a canonical symplectic form compatible with multiplication \cite[\S5]{mackenzie-xu}.
Conversely, every symplectic groupoid $\G \rightrightarrows X$ induces a Poisson structure $\sigma$ on $X$ such that $\Lie(\G) = (T^*X)_\sigma$.
Poisson manifolds are then viewed as infinitesimal versions of symplectic groupoids, and we say that $\G$ an \emph{integration} $X$, in analogy with the integration of a Lie algebra to a Lie group.
Not all Poisson manifolds are integrable, but they are all integrable by a symplectic \emph{local} groupoid (the axiomatization of a neighbourhood of the identity section of a symplectic groupoid), as first shown by Karasev \cite{karasev-analogues} and Weinstein \cite{weinstein-groupoids, CDW}. This was also reproved later by Cattaneo--Felder \cite{cattaneo-felder-groupoids} using a Poisson sigma model, which led to a complete criterion for global integrability by Crainic--Fernandes \cite{crainic-fernandes-lie, crainic-fernandes-poisson}.

A symplectic local groupoid is, in particular, a \emph{symplectic realization} \cite{weinstein-local, CDW}, i.e.\ a symplectic manifold $M$ together with a surjective Poisson submersion $M \to X$ and a Lagrangian section $X \hookrightarrow M$.
Conversely, all symplectic realizations are symplectic local groupoid, and this structure is essentially unique after restricting to a neighbourhood of the identity section \cite[III \S1]{CDW}.
The symplectic realization of a Poisson manifold $X$ can then be thought of as a canonical desingularization of $X$, where the foliation by symplectic leaves on $X$ has been resolved to a single leaf $M$.

There is an analogous theory of Poisson manifolds in the holomorphic category, which has gained considerable interest in recent years, in particular through its relationship with generalized geometry, deformation theory, and classification problems.
In particular, the question of which holomorphic Poisson manifold admits an integration by a holomorphic symplectic groupoid has been solved \cite{lgsx-holomorphic, lgsx}, and one of the consequences is an analogue of the Karasev--Weinstein theorem: every holomorphic Poisson manifold is integrable by a holomorphic symplectic local groupoid (see also \cite{BX}).
As examples of holomorphic symplectic manifolds (typically non-compact), it is natural to ask if they admit hyperk\"ahler metrics.

The archetypal example is the dual $\g^*$ of a Lie algebra with its Kirillov--Kostant--Souriau Poisson structure, which is integrated by $T^*G$, where $G$ is any Lie group integrating $\g$. 
If $\g$ is complex semisimple, $T^*G$ admits a hyperk\"ahler structure discovered by Kronheimer \cite{kronheimer-cotangent} as a moduli space of solutions to the one-dimensional reduction of the anti-self-dual Yang--Mills equations (Nahm's equations).
Similarly, if $X$ is endowed with the zero Poisson structure
$\sigma = 0$, then $T^*X$ is an integration, and we have the
Feix--Kaledin metric \cite{feix-cotangent, kaledin-2, kaledin-1} on a neighbourhood of the zero section, which is the identity section when viewed as a groupoid.
At the other extreme, if $\sigma$ is non-degenerate and $X$ is compact, then $X$ itself is hyperk\"ahler, and hence so is its groupoid $X \times X$.
From these examples, it is natural to ask:

\begin{question}
Does every holomorphic symplectic groupoid over a K\"ahler manifold have a hyperk\"ahler structure on a neighbourhood of its identity section?
\end{question}

We will show that the answer is `yes' when the base $X$ is of complex dimension two and is compact.
This can be viewed as another instance where real dimension four plays a special r\^ole.

To prove this existence result, we first formulate the problem in any dimension to a problem of finding special deformations of Poisson structures on $X$  adapted from Hitchin's unobstructedness theorem \cite{hitchin-def}.
If they exist, these deformations can be lifted to the holomorphic symplectic groupoid to construct the twistor space (Theorem \ref{groupoid-theorem}).
We then show existence in the two dimensional case (Theorem \ref{surface-theorem}).

\subsection{Statement of results}

Recall that the complex structure of a holomorphic symplectic manifold $(M, \Omega_0)$ is uniquely determined by the symplectic form $\Omega_0$, since $T^{0, 1}M = \ker(\Omega_0)$.  
Hence, we define a \defn{deformation of holomorphic symplectic structures} to be a family of complex 2-forms $\Omega_\zeta$ on $M$, depending holomorphically on $\zeta$ in a neighbourhood of $0$ in $\C$, such that each $\Omega_\zeta$ is holomorphic symplectic with respect to some (unique) complex structure.
We say that the family is \defn{real-analytic} if it is given in holomorphic coordinates of $(M, \Omega_0)$ by real-analytic functions.
Also, since a holomorphic symplectic form determines its complex structure, a hyperk\"ahler structure $(g, I, J, K)$ can be identified with its triple $(\omega_I, \omega_J, \omega_K)$ of K\"ahler forms.

Our first result is a characterization of hyperk\"ahler metrics near Lagrangian submanifolds in terms of deformations of holomorphic symplectic structures:

\begin{maintheorem}\label{lagrangian-theorem}
Let $(M, \Omega_0)$ be a holomorphic symplectic manifold and $X$ a complex Lagrangian submanifold. 
Suppose that there is a real-analytic deformation of holomorphic symplectic structures $\Omega_\zeta$ of $(M, \Omega_0)$ such that $\iota^* \Omega_\zeta = 2i\zeta\omega$ for some K\"ahler form $\omega$ on $X$, where $\iota : X \hookrightarrow M$ is the inclusion map.
Then, there is a hyperk\"ahler structure $(\omega_I, \omega_J, \omega_K)$ on a neighbourhood of $X$ in $M$ such that $\Omega_0 = \omega_J + i \omega_K$ and $\iota^* \omega_I = \omega$.
\end{maintheorem}

The hyperk\"ahler structure is obtained by constructing the twistor space by gluing two halves together, as in \cite{feix-cotangent}.
We review this method in \S\ref{twistor} and prove Theorem \ref{lagrangian-theorem} in \S\ref{lagrangian-theorem-proof}.

\begin{remark}
Conversely, every hyperk\"ahler structure $(\omega_I, \omega_J, \omega_K)$ can be obtained in this way using the twistor family (see \S\ref{twistor}) of holomorphic symplectic forms
\begin{equation}\label{twistor-family}
\Omega_\zeta \coloneqq (\omega_J + i \omega_K) + 2i\zeta \omega_I + \zeta^2(\omega_J - i \omega_K).
\end{equation}
\end{remark}

\begin{remark}
More generally, if $\omega$ is a pseudo-K\"ahler form on $X$ of signature $(p, q)$, we get a pseudo-hyperk\"ahler structure of signature $(2p, 2q)$ (Lemma \ref{signature-lemma}).
\end{remark}

\begin{example}\label{feix-kaledin-example}
If $X$ is a complex manifold with a real-analytic K\"ahler
form $\omega$, we recover the Feix--Kaledin hyperk\"ahler
metric \cite{feix-cotangent, kaledin-2, kaledin-1} on a
neighbourhood of the zero section of $M = T^*X$ by taking $\Omega_\zeta \coloneqq \Omega_0 + 2i\zeta \pi^*\omega$, where $\pi : T^*X \to X$ is the bundle map and $\Omega_0$ is the canonical holomorphic symplectic form (see \S\ref{trivial-poisson-structure}).
This is known as the \emph{twisted cotangent bundle}, and appears, for example, in Donaldson's reformulation of K\"ahler geometry \cite[\S2]{donaldson-monge}, as also explained in \cite[\S2]{bgz}.
The complex structure induced by $\Omega_\zeta$ is not the original one, but is biholomorphic to the affine bundle modeled on $T^*X$ defined by the K\"ahler class $[\omega] \in H^1(X, T^*X)$ (see \cite[Appendix]{abasheva} and \cite{GW}).
\end{example}

\begin{remark}[Deformations of hyperk\"ahler structures]
Theorem \ref{lagrangian-theorem} implies that any complex Lagrangian submanifold in a hyperk\"ahler manifold defines a deformation of hyperk\"ahler structures. 
Indeed, \eqref{twistor-family} satisfies the condition of the theorem, but so does $\Omega_{t\zeta}$ for any $t > 0$. 
Hence, there is a family of K\"ahler forms $\omega_I(t)$ on neighbourhoods of $X$, such that $(\omega_I(t), \omega_J, \omega_K)$ is hyperk\"ahler and $\iota^*\omega_I(t) = t \iota^*\omega_I$ for all $t > 0$.  
For example, by applying this idea to Kronheimer's hyperk\"ahler structure on the cotangent bundle of a complex reductive group \cite{kronheimer-cotangent}, the zero section induces the one-parameter family of hyperk\"ahler structures obtained by varying the length of the interval on which Nahm's equations are defined.
\end{remark}

We now turn our attention to holomorphic Poisson geometry.
Recall that a \defn{holomorphic Poisson manifold} is a complex manifold $X$ together with a holomorphic bivector field $\sigma$ such that the bilinear operator $\{f, g\} \coloneqq \sigma(df, dg)$ acting on holomorphic functions $f, g$ satisfies the Jacobi identity.
A \defn{holomorphic symplectic groupoid} (see e.g.\ \cite[\S3]{lgsx}) is a holomorphic Lie groupoid $\G \rightrightarrows X$ together with a holomorphic symplectic form $\Omega$ such that the graph of the multiplication map is Lagrangian in $\G \times \G \times \G^-$, where $\G^-$ is $\G$ with the opposite symplectic structure $-\Omega$.
It follows that $X$ is endowed with a canonical holomorphic Poisson structure such that the source map $s : \G \to X$ is a holomorphic Poisson submersion.
The identity section is always a complex Lagrangian submanifold \cite[II Proposition 1.1]{CDW}, and we wish to use it with Theorem \ref{lagrangian-theorem}.
Since we are only considering a neighbourhood of the identity section, we can in fact generalize the discussion to symplectic realizations in the following sense.

\begin{definition}[{Weinstein \cite[\S7]{weinstein-local} \cite[III \S1]{CDW}}]\label{symplectic-realization}
A \defn{symplectic realization} of a holomorphic Poisson manifold $X$ is a holomorphic symplectic manifold $M$ containing $X$ as a complex Lagrangian submanifold together with a surjective holomorphic Poisson submersion $s : M \to X$ such that $s \circ \iota = \mathrm{id}_X$, where $\iota : X \hookrightarrow M$ is the inclusion map.
\end{definition}

\begin{remark}
In \cite[III Definition 1.1]{CDW}, this is called a \emph{strict} symplectic realization.
\end{remark}

For example, a neighbourhood of the identity section of a holomorphic symplectic groupoid together with its source map is a symplectic realization.  
Conversely, any symplectic realization has the structure of a symplectic local groupoid, possibly after restricting to a neighbourhood of $X$ in $M$ \cite[III Theorem 1.2]{CDW}. 
Also, every holomorphic Poisson manifold has a symplectic realization in this sense \cite{BX}.

The next result shows that a hyperk\"ahler structure on a symplectic realization can be constructed if we can solve certain differential equations on the Poisson manifold.

\begin{maintheorem}\label{groupoid-theorem}
Let $(X, \sigma)$ be a (not-necessarily compact) holomorphic Poisson manifold.
Suppose that there is a real-analytic family $\omega(\zeta) = \sum_{n = 1}^\infty \omega_n \zeta^n$ of $(1, 1)$-forms on $X$ solving the equation
\begin{equation}\label{srg86qgv}
d \omega + \tfrac{1}{2}\d i_\sigma(\omega
\wedge \omega) = 0,
\end{equation}
for $\zeta$ in a neighbourhood $U$ of $0$ in $\C$, such that 
\begin{enumerate}
\item[(a)] $\omega_1$ is a K\"ahler form,
\item[(b)] $1 - \phi_\zeta \bar{\phi}_\zeta$ is invertible for all $\zeta \in U$, where $\phi_\zeta \coloneqq -\sigma \circ \omega(\zeta) : T_\C X \to T_\C X$,
\item[(c)] $i_\sigma(\omega_{2n} \wedge \omega_1) = 0$ and $\omega_{2n + 1} = 0$ for all $n \ge 1$.
\end{enumerate}
Then, any symplectic realization $(M, \Omega_0)$ of $(X, \sigma)$ has a hyperk\"ahler structure $(\omega_I, \omega_J, \omega_K)$ on a neighbourhood of $X$ such that $\Omega_0 = \omega_J + i \omega_K$ and $\iota^* \omega_I = \omega_1$.
\end{maintheorem}

Equation \eqref{srg86qgv} comes from Hitchin's method \cite{hitchin-def} for constructing deformations of holomorphic Poisson structures, where $\phi_\zeta$ is shown to solve the Maurer--Cartan equation.
This will be explained in details in \S\ref{groupoid-theorem-proof}.
The conditions on the coefficients $\omega_n$ ensure that this deformation can be lifted to a deformation of holomorphic symplectic structures $\Omega_\zeta$ on $M$ in such a way that $\iota^*\Omega_\zeta = 2i \zeta \omega_1$, as required for Theorem \ref{lagrangian-theorem}.
More precisely, we define
\[
\Omega_\zeta \coloneqq \Omega_0 + s^* \beta(i\zeta) - t^* \beta(-i\zeta),
\]
where $\beta(\zeta) = \omega(\zeta) + \tfrac{1}{2} i_\sigma(\omega(\zeta) \wedge \omega(\zeta))$ (which can be viewed as a $B$-field) and $t$ is a target map.
If $X$ is compact K\"ahler, there always exists $\omega(\zeta)$ satisfying (a) and (b), but (c) imposes additional constraints which we can only obtain in special cases.

\begin{example}
We recover the Feix--Kaledin hyperk\"ahler structure even more easily from Theorem \ref{groupoid-theorem} with $\sigma = 0$ and $\omega(\zeta) \coloneqq \zeta \omega_1$ for any real-analytic K\"ahler form $\omega_1$.
We will show that our hyperk\"ahler structure coincides with the Feix--Kaledin one in \S\ref{trivial-poisson-structure}.
\end{example}

\begin{example}
Let $X$ be a hyperk\"ahler manifold with K\"ahler forms $(\omega_I, \omega_J, \omega_K)$ and let $\sigma$ be the inverse of the holomorphic symplectic structure $\Omega = \omega_J + i \omega_K$. 
Then, $\omega(\zeta) = \zeta \omega_I$ satisfies the conditions of Theorem \ref{groupoid-theorem}. 
Indeed, $\tfrac{1}{2}i_{\sigma}(\omega_I \wedge \omega_I) = \omega_I \sigma \omega_I = \tfrac{1}{4} \omega_I (\omega_J^{-1} - i \omega_K^{-1}) \omega_I = \tfrac{1}{4}(-\omega_J + i \omega_K) = -\frac{1}{4} \bar{\Omega}$, using that $\omega_J^{-1}\omega_I = K$ so $\omega_J^{-1}\omega_I = -\omega_I^{-1}\omega_J$ and similarly $\omega_K^{-1} \omega_I = - \omega_I^{-1} \omega_K$ (cf.\ \cite[Example 5.7]{gualtieri-ham}).
Hence, both terms of $d\omega + \tfrac{1}{2} \d i_\sigma(\omega \wedge \omega)$ are zero.
Now, the pair groupoid $X \times X$ is a symplectic realization of $X$, with symplectic form $(\Omega, -\Omega)$, source map $s(x, y) = x$, and identity section $\iota(x) = (x, x)$.
Then, Theorem \ref{groupoid-theorem} recovers the hyperk\"ahler structure on $X \times X$ with complex structures $(I, I)$, $(J, -J)$, $(K, -K)$.
\end{example}

Finally, we show that the equations in Theorem \ref{groupoid-theorem} can be solved in the two-dimensional case.

\begin{maintheorem}\label{surface-theorem}
Let $(X, \sigma)$ be a compact holomorphic Poisson surface together with a real-analytic K\"ahler form $\omega$ and let $s : (M, \Omega) \to (X, \sigma)$ be a symplectic realization. 
Then, there is a hyperk\"ahler structure $(\omega_I, \omega_J, \omega_K)$ on a neighbourhood of $X$ in $M$ such that $\Omega = \omega_J + i \omega_K$ and $\iota^*\omega_I = \omega$.
\end{maintheorem}

\begin{remark}
Real-analyticity of the K\"ahler form is a necessary condition since the twistor correspondence \cite[\S3(F)]{HKLR} implies that all hyperk\"ahler structures are real-analytic.  
On the other hand, every K\"ahler class has a real-analytic representative \cite[Proposition 1.7]{GW}.
\end{remark}

\begin{example}
For any two-dimensional compact holomorphic symplectic manifold $X$ with a real-analytic K\"ahler form $\omega$, Theorem \ref{surface-theorem} shows that there is a hyperk\"aler structure $(\omega_I, \omega_J, \omega_K)$ on a neighbourhood of the diagonal in $X \times X$ such $\iota^*\omega_I = \omega$.
There is, of course, already a hyperk\"ahler structure on $X$ by Yau's theorem \cite{yau} and hence on $X \times X$, but the restriction of the first K\"ahler form to the diagonal is not equal to $\omega$ in general, so the two hyperk\"ahler structures are different.
\end{example}

\subsection{Future work}

The next natural step is, of course, the study of examples. 
There is a complete classification of compact holomorphic Poisson surfaces up to birational equivalences derived from the Enriques--Kodaira classification, so there is a wealth of explicit examples to study (see e.g.\ Pym's overview \cite{pym-constructions}).
In particular, there is a large family of Poisson structures  on $\CP^2$ obtained by choosing a cubic curve to specify the vanishing locus of the Poisson bivector field.
Using the Fubini--Study metric, our result then implies the existence of a canonical hyperk\"ahler manifold associated to any elliptic curve, and which deforms the complete hyperk\"ahler metric on $T^*\CP^2$ found by Calabi \cite{calabi} (which can be viewed as coming from the zero Poisson structure).
An interesting question for future work is whether these new metrics extend to the whole groupoid and are complete.
Since hyperk\"ahler structures are real-analytic, the extension is unique.

\subsection{Acknowledgments} 
I thank Marco Gualtieri for helpful discussions, and the anonymous referee for useful comments.
This work was supported by an NSERC Postdoctoral Fellowship and additional support from the FRQNT.

\section{Twistor spaces by gluing}
\label{twistor}

We now recall and clarify a general construction of hyperk\"ahler metrics by gluing two halves of a twistor space.
This is the idea used by Feix \cite{feix-cotangent} for the cotangent bundle of a K\"ahler manifold.
It also appears in many works, such as Deligne's construction of the twistor space of the Hitchin moduli space, as explained by Simpson \cite[\S4]{simpson-hodge}, and Atiyah--Hitchin's description of the hyperk\"ahler metric of the monopole moduli space \cite[Ch.\ 5]{atiyah-hitchin}, for instance.

\subsection{The hyperk\"ahler twistor correspondence}

In this subsection, we will first recall Hitchin--Karlhede--Lindstr\"{o}m--Ro\v{c}ek's adaptation of Penrose's non-linear twistor theory \cite{penrose} for the construction of hyperk\"ahler metrics \cite[\S3(F)]{HKLR}.
The idea is encode a hyperk\"ahler structure into purely holomorphic data on an auxiliary space.

We work with pseudo-hyperk\"ahler manifolds first; the positive definiteness of the metric is a separate condition. 
Recall that a \defn{pseudo-hyperk\"ahler manifold} is a smooth manifold $M$ with a pseudo-Riemannian metric $g$ and complex structures $I, J, K$ which are pseudo-K\"ahler with respect to $g$ and satisfy $IJK = -1$. 
The corresponding pseudo-K\"ahler forms will be denoted $\omega_I, \omega_J, \omega_K$ and, since they determine $I, J, K$ and $g$, a pseudo-hyperk\"ahler structure will be identified with its triple $(\omega_I, \omega_J, \omega_K)$ of symplectic forms.

Let $S^2 \s \R^3$ be the unit two-sphere endowed with its standard complex structure $\mathsf{I}_xy = x \times y$ for $x \in S^2$ and $y \in T_xS^2 \s \R^3$. 
For $x \in S^2$, let $I_x \coloneqq x_1 I + x_2 J + x_3 K$ and $\omega_x \coloneqq x_1 \omega_I + x_2 \omega_J + x_3 \omega_K$.
Then, $(g, I_x, \omega_x)$ is a pseudo-K\"ahler structure for all $x \in S^2$.
Moreover, the space $Z \coloneqq M \times S^2$ becomes a complex manifold by endowing it with the almost complex structure $(I_x, \mathsf{I})$ at $(p, x) \in M \times S^2$, and the projection $\pi : Z \to S^2$ is a surjective holomorphic submersion.

Let $x \in S^2$ and choose $y, z \in \R^3$ so that $(x, y, z)$ is a positively oriented orthonormal basis.  
Then, the complex two-form $\omega_y + i\omega_z$ is a holomorphic symplectic form with respect to $I_x$. 
Moreover, by viewing $y$ as a tangent vector to $x$, the tensor product $(\omega_y + i \omega_z) \otimes y \in \Exterior^2(\ker d\pi)^* \otimes T_xS^2$ is independent of the choice of $(y, z)$, and defines a global \emph{holomorphic} section $\Omega$ of $\Exterior^2(\ker d\pi)^* \otimes \pi^*TS^2$.
Also, the map $\tau : Z \to Z : (p, x) \mto (p, -x)$ is a real structure (i.e.\ an anti-holomorphic involution) covering the antipodal map on $S^2$ and such that $\tau^*\bar{\Omega} = -\Omega$.
Finally, each point $p$ of $M$ induces a holomorphic section $x \mto (p, x)$ of $\pi$, which is $\tau$ invariant and has normal bundle isomorphic to $\O(1)^{\oplus 2n}$.

Now, identify $S^2$ with $\CP^1$ via the biholomorphism
\[
\CP^1 \too S^2, \quad \zeta \mtoo 
\left(
\frac{1 - |\zeta|^2}{1 + |\zeta|^2},
i \frac{\bar{\zeta} - \zeta}{1 + |\zeta|^2},
- \frac{\zeta + \bar{\zeta}}{1 + |\zeta|^2}
\right).
\]
Then, the holomorphic section $\Omega$ defined above is of the form
\[
\Omega = \left((\omega_J + i \omega_K) + 2i \zeta \omega_I + \zeta^2 (\omega_J - i \omega_K) \right) \otimes \tfrac{i}{2} \tfrac{\d}{\d \zeta}.
\]
Upon identifying $T\CP^1$ with $\O(2)$, we arrive at:

\begin{definition}\label{twistor-space}
A \defn{hyperk\"ahler twistor space} is a complex manifold $Z$ of dimension $2n + 1$, together with
\begin{enumerate}
\item[(a)] a surjective holomorphic submersion $\pi : Z \to
\CP^1$,
\item[(b)] a global holomorphic section $\Omega$ of $\Exterior^2(\ker d\pi)^* \otimes \pi^*\O(2)$ which restricts to a holomorphic symplectic form on each fibre $Z_\zeta \coloneqq \pi^{-1}(\zeta)$, and
\item[(c)] a real structure $\tau : Z \to Z$ covering the antipodal map $\rho : \CP^1 \to \CP^1$ and such that $\tau^*\bar{\Omega} = -\Omega$.
\end{enumerate}
A \defn{twistor line} is a holomorphic section of $\pi$ whose normal bundle is isomorphic to $\O(1)^{\oplus 2n}$. 
A twistor line $s$ is \defn{real} if $\tau \circ s = s \circ \rho$.
\end{definition}

The twistor method says that, conversely, this data is enough to construct a hyperk\"ahler manifold.

\begin{theorem}[{Hitchin--Karlhede--Lindstr\"{o}m--Ro\v{c}ek \cite[\S3(F)]{HKLR}}]\label{7d5uk40i}
Let $(Z, \pi, \Omega, \tau)$ be a hyperk\"ahler twistor space. 
Then, the set $\M$ of real twistor lines is a smooth manifold endowed with a pseudo-hyperk\"ahler structure $(\omega_I, \omega_J, \omega_K)$. 
Moreover, for each $\zeta \in \CP^1 \setminus \{\infty\}$, the evaluation map $\mathrm{ev}_\zeta : \M \to Z_\zeta$ is a local biholomorphism with respect to $I_\zeta$ on $\M$ and a complex symplectomorphism, i.e.\ 
\[
\mathrm{ev}_\zeta^*\Omega_\zeta = (\omega_J + i\omega_K) + 2i\zeta \omega_I + \zeta^2 (\omega_J - i \omega_K),
\]
where $\Omega_\zeta$ is the $2$-form on $Z_\zeta$ defined by $\Omega = \Omega_\zeta \otimes \tfrac{i}{2} \frac{\d}{\d \zeta}$. \qed
\end{theorem}

In particular, if we find a single real twistor line $s : \CP^1 \to Z$ passing through a point $x \in Z_0$, the theorem says that $s$ lives in a $4n$-dimensional family of real twistor lines and that there is a hyperk\"ahler structure on a neighbourhood of $x$ in $Z_0$. 
If, more generally, we have a submanifold $X \s Z_0$ and a family real twistor lines $s_x \in \M$ parametrized by $X$ such that $s_x(0) = x$ for all $x \in X$, then there is a hyperk\"ahler structure on a neighbourhood of $X$ in $Z_0$ compatible with its holomorphic symplectic structure. 
This is the idea behind Feix's proof \cite{feix-cotangent}, which summarize as follows:

\begin{corollary}\label{family-of-twistor-lines}
Let $(Z, \pi, \Omega, \tau)$ be a hyperk\"ahler twistor space and let $\M$ be the space of real twistor lines.
Suppose that there is a smooth submanifold $X \s Z_0$ and a smooth map $X \to \M : x \mto s_x$ such that $s_x(0) = x$ for all $x \in X$.  
Then, there is a unique pseudo-hyperk\"ahler structure $(\omega_I, \omega_J, \omega_K)$ on a neighbourhood of $X$ in $Z_0$
such that
\[
\mathrm{ev}_\zeta^* \Omega_\zeta
=
\mathrm{ev}_0^*( (\omega_J + i\omega_K) + 2i \zeta \omega_I + \zeta^2 (\omega_J - i \omega_K))
\]
for all $\zeta \ne \infty$, where $\Omega_\zeta$ is defined by $\Omega = \Omega_\zeta \otimes \frac{i}{2} \frac{\d}{\d \zeta}$. 
\end{corollary}

\begin{proof}
Since $\mathrm{ev}_0 : \M \to Z_0$ is a local diffeomorphism and $s : X \to \M : x \mto s_x$ is a section of $\mathrm{ev}_0$, $s$ is an embedding.
Hence, by restricting to a sufficiently small neighbourhood $U$ of $s(X)$ in $\M$, $\mathrm{ev}_0 : U \to Z_0$ is a diffeomorphism onto its image, so we can pushforward the pseudo-hyperk\"ahler structure.
\end{proof}

We now make a few observations which will be useful to find twistor lines.

\begin{lemma}\label{zgd97fbr}
Let $(Z, \pi, \Omega, \tau)$ be a $(2n + 1)$-dimensional hyperk\"ahler twistor space.
Then, the normal bundle of any holomorphic section of $\pi : Z \to \CP^1$ has degree $2n$.
\end{lemma}

\begin{proof}
The normal bundle $N$ of a holomorphic section $s$ can be identified with $s^*\ker d\pi$, so $\Omega$ restricts to a section of $\Exterior^2N^* \otimes \O(2)$ and hence $\Omega^n$ gives a non-vanishing section of $\Exterior^{2n}N^* \otimes \O(2n)$, so $\det(N) \cong \O(2n)$.
\end{proof}

To get a twistor line, we need moreover that $N \cong \O(1)^{\oplus 2n}$, and this will be obtained from the following lemma.

\begin{lemma}[{cf.\ \cite[p.\ 41]{feix-cotangent}}]\label{yaed55v1}
Let $E \to \CP^1$ be a holomorphic vector bundle of rank $m$.
Then, $E \cong \O(1)^{\oplus m}$ if and only if $\deg(E) = m$ and there exist global holomorphic sections $s_1, \ldots, s_m$ of $E$ which are linearly independent in at least one of the fibres and each $s_i$ has a zero.
\end{lemma}

\begin{proof}
The forward direction is clear. 
Conversely, by Grothendieck's theorem, we have $E = \bigoplus_{i = 1}^m \O(k_i)$ for some $k_i \in \Z$ such that $\sum_{i = 1}^m k_i = m$, so it suffices to show that $k_i > 0$ for all $i$.  
Suppose that $k_i \le 0$ for some $i$. If $k_i = 0$, then for all $j$, the projection of $s_j$ to $\O(k_i) = \O$ is zero since $s_j$ has a zero. 
This is also true if $k_i < 0$ since then $\O(k_i)$ has no non-zero section.
In particular, for all $\zeta \in \CP^1$, $s_1(\zeta), \ldots, s_m(\zeta)$ are contained in an $(m - 1)$-dimensional subspace of $E_\zeta$, contradicting that they linear independent in one of the fibres.
\end{proof}

\subsection{The gluing construction}
\label{gluing-construction}

If $(Z, \pi, \Omega, \tau)$ is a hyperk\"ahler twistor space, the real structure $\tau$ identifies $Z|_{\CP^1\setminus\{\infty\}}$ and $Z|_{\CP^1\setminus\{0\}}$. 
Hence, $Z$ can be recovered from its restriction to $\C \s \CP^1$ and the gluing map $\tau : Z|_{\C^*} \to Z|_{\C^*}$.
In the context of deformation theory, we obtain deformation spaces $Z \to \C$ over arbitrarily small neighbourhoods of $0$ in $\C$, so it is useful to formulate this gluing principle slightly more generally.

For $\e, \delta > 0$, define the discs, circle, and annulus
\begin{align*}
D_\e &\coloneqq \{\zeta \in \C : |\zeta| < \e\} \\
\bar{D}_\e &\coloneqq \{\zeta \in \C : |\zeta| \le \e\} \\
S^1_\e &\coloneqq \{\zeta \in \C : |\zeta| = \e\} \\
A_{\e, \delta} &\coloneqq \{\zeta \in \C : \frac{\e^2}{\e + \delta} < |\zeta| < \e +
\delta\}.
\end{align*}
Then, the antipodal map $\zeta \mto -\zeta$ on $S^1_\e$ extends uniquely to a real structure
\begin{equation}\label{antipodal}
\rho_\e : \C^* \too \C^*, \quad \rho_\e(\zeta) = -\e^2 / \bar{\zeta},
\end{equation}
which restricts to $A_{\e, \delta} \to A_{\e, \delta}$ for all $\delta > 0$.
Moreover, $\CP^1$ can be constructed by gluing two copies of $D_{\e + \delta}$ by $\rho_\e: A_{\e, \delta} \to A_{\e, \delta}$, where the second copy of $D_{\e + \delta}$ has the opposite complex structure.
Twistor spaces can be built similarly by deformation spaces over small discs in $\C$:

\begin{proposition}\label{local-twistor-space}
Let $Z$ be a complex manifold together with
\begin{enumerate}
\item[(a)] a holomorphic submersion $\pi : Z \to \C$ such that $\bar{D}_\e \s \pi(Z)$ for some $\e > 0$,
\item[(b)] a global holomorphic section $\Omega$ of $\Exterior^2(\ker d\pi)^* \to Z$ which restricts to a holomorphic symplectic form on each fibre of $\pi$,
\item[(c)] a real structure $\tau : U \to U$ covering $\rho_\e$ and satisfying $\tau^* \bar{\Omega} = \e^2 \Omega / \pi^2$, where $U \s Z$ is an open set such that $S^1_\e \s \pi(U)$ and such that $\tau$ extends continuously to $\bar{U} \setminus Z_0$.
\end{enumerate}
Then, this data defines a hyperk\"ahler twistor space $\mathcal{Z}$ extending $(Z, \pi, \Omega, \tau)$, where $\mathcal{Z}$ is defined by gluing two copies of $V \coloneqq U \cup \pi^{-1}(D_\e)$ via $\tau : U \to U$,
\[
\mathcal{Z} \coloneqq (V \sqcup V) / \{z \sim \tau(z)\}.
\]
A section $s$ of $\mathcal{Z}$ is equivalent to a pair of sections $s_\pm : D_{\e + \delta} \to U \cup \pi^{-1}(D_\e)$ of $\pi$, for some $\delta > 0$, such that $\tau(s_+(\zeta)) = s_-(\rho_\e(\zeta))$ for all $\zeta \in s_+^{-1}(U)$. 
The section is real if and only if $s_+ = s_-$. 
The normal bundle of $s$ in $\mathcal{Z}$ is isomorphic to the gluing of the bundles $s_+^*(\ker d\pi)$ and $s_-^*(\ker d\pi)$ over $D_{\e + \delta}$ via $d\tau$.
\end{proposition}

\begin{lemma}
The space $\mathcal{Z}$ is Hausdorff.
\end{lemma}

\begin{proof}
It suffices to show that for all $x, y \in \bar{U} \cap V \setminus U$ there are neighbourhoods $V_x$ and $V_y$ of $x$ and $y$ in $V$ such that $\tau(V_x \cap U) \cap V_y = \emptyset$. 
If $x \in Z_0$, we can take $0 < \delta < \e$ small enough so that $|\pi(y)| + \delta < \frac{1}{\delta}$ and let $V_x = \pi^{-1}(D_\delta)$ and $V_y = \pi^{-1}(A_{|\pi(y)|, \delta}) \cap V$. 
Then, if $z \in \tau(V_x \cap U) \cap V_y$, we have $\frac{1}{\delta} < |\pi(z)| < |\pi(y)| + \delta$, contradicting the choice of $\delta$. 
Hence, we may assume that $x \in \bar{U} \setminus Z_0$ so that $\tau(x)$ is well-defined. 
We claim that $\tau(x) \notin V$. 
First, note that $|\pi(\tau(x))| > \e$ since $\pi(\tau(x)) = \rho_\e(\pi(x))$ and $|\pi(x)| < \e$. 
Hence, if $\tau(x) \in V$, then $\tau(x) \in U$ so $x = \tau(\tau(x)) \in U$, a contradiction. 
So $\tau(x) \notin V$ and hence $\tau(x) \neq y$. 
Since $Z$ is Hausdorff, there are neighbourhoods $W_x$ and $W_y$ of $\tau(x)$ and $y$ in $Z$ such that $W_x \cap W_y = \emptyset$ so we can take $V_x = V \cap \tau^{-1}(W_x)$ and $V_y = V \cap W_y$.
\end{proof}

\begin{proof}[Proof of Proposition \ref{local-twistor-space}]
Since $\tau$ is anti-holomorphic, it follows that $\mathcal{Z}$ is a complex manifold by endowing the second copy of $V$ with the opposite complex structure. 
Moreover, by identifying $\CP^1$ with the gluing of two copies of $D_{\e + \delta}$ by $\rho_\e$, we get a surjective holomorphic submersion $\pi : \mathcal{Z} \to \CP^1$. 
The real structure $\mathcal{T} : \mathcal{Z} \to \mathcal{Z}$ is defined to be the anti-holomorphic involution which descends from the map $V \sqcup V \to V \sqcup V$ exchanging the two copies of $V$. 

Now, consider the local section $\Psi \coloneqq \Omega \otimes \frac{i}{2} \frac{\d}{\d \zeta}$ of $\Exterior^2(\ker d\pi)^* \otimes T\CP^1$.  
Let $\tau_\zeta : U \cap Z_\zeta \to U \cap Z_{\rho_\e(\zeta)}$ be the restriction of $\tau$ on the fibres. 
Then, the condition $\tau^* \bar{\Omega} = \e^2 \Omega / \pi^2$ is equivalent to $\tau_\zeta^* \bar{\Omega}_{\rho_\e(\zeta)} = \e^2 \Omega_\zeta / \zeta^2$, so we have
\[
\tau^*\bar{\Psi} = \tau_\zeta^*\bar{\Omega}_{\rho_\e(\zeta)} \otimes d\rho_\e(-\tfrac{i}{2}\tfrac{\d}{\d\overline{\zeta}}\big|_{\rho_\e(\zeta)}) =  (\tfrac{\e^2}{\zeta^2} \Omega_\zeta) \otimes (-\tfrac{i}{2} \tfrac{\zeta^2}{\e^2} \tfrac{\d}{\d\zeta}) = - \Psi.
\]
Hence, $\Psi$ glues to a global holomorphic section of $\Exterior^2 (\ker d\pi)^* \otimes T\CP^1 \to \mathcal{Z}$ such that $\mathcal{T}^*\bar{\Psi} = -\Psi$.  

The final statements about holomorphic sections of $\mathcal{Z}$ follow directly from the definition of $\mathcal{Z}$.
\end{proof}

One way to think of the twistor correspondence is that we start with a deformation of holomorphic symplectic structures $\Omega_\zeta$ on a holomorphic symplectic manifold $(M, \Omega_0)$, which is not necessarily quadratic in $\zeta$.
If the deformation space $Z = M \times \C$ has a real structure $\tau$ enabling us to glue it to a twistor space as in Proposition \ref{local-twistor-space}, then there is a family of diffeomorphisms $f_\zeta \coloneqq \mathrm{ev}_\zeta \circ \mathrm{ev}_0^{-1}: M \to M$ with the very special property that $f_\zeta^*\Omega_\zeta = \Omega_0 + 2i\zeta \omega_I + \zeta^2 \bar{\Omega}_0$ for some K\"ahler form $\omega_I$ whose metric is hyperk\"ahler.
In some situations, such as Hitchin's construction of the Gibbons--Hawking gravitational multi-instantons \cite{hitchin-polygons}, this family of diffeomorphisms $f_\zeta$ can be computed explicitly by solving algebraic equations (i.e.\ finding all real twistor lines), which then produces an explicit description of the hyperk\"ahler metric.

\section{Hyperk\"ahler metrics near Lagrangian submanifolds}
\label{lagrangian-theorem-proof}

The goal of this section is prove Theorem \ref{lagrangian-theorem}.

Let $(M, \Omega_0)$ be a holomorphic symplectic manifold, $X \s M$ a complex Lagrangian submanifold, and $\iota : X \hookrightarrow M$ the inclusion map.
Let $\Omega_\zeta$ be a real-analytic family of holomorphic symplectic forms on $M$ depending holomorphically on $\zeta$ in a disc $D_r = \{\zeta \in \C : |\zeta| < r\}$ for some $r > 0$, where the complex structure determined by $\Omega_\zeta$ may also be varying.
More precisely, in holomorphic coordinates $(z_k = x_{2k - 1} + i x_{2k})_{k = 1}^{m}$ on $M$ (in the complex structure determined by $\Omega_0$), we have $\Omega_\zeta = \Omega_{ij}(x_1, \ldots, x_{2m}, \zeta) dx^i \wedge dx^j$, where $\Omega_{ij}$ are real-analytic complex-valued functions which are holomorphic in $\zeta$. Suppose also that there is a pseudo-K\"ahler form $\omega$ on $X$ such that $\iota^*\Omega_\zeta = 2i\zeta\omega$ for all $\zeta \in D_r$.

We will construct a pseudo-hyperk\"ahler structure $(\omega_I, \omega_J, \omega_K)$ on a neighbourhood of $X$ in $M$ such that $\Omega_0 = \omega_J + i\omega_K$ and $\iota^*\omega_I = \omega$, and which is positive definite if $\omega$ is.
It will be obtained by constructing the twistor space as in Proposition \ref{local-twistor-space}.

Let $Z = M \times D_r$ with the almost complex structure $(I_\zeta, i)$ at $(x, \zeta) \in M \times D_r$, where $I_\zeta$ is the complex structure corresponding to $\Omega_\zeta$.

\begin{lemma}
The almost complex structure on $Z = M \times D_r$ is integrable.
\end{lemma}

\begin{proof}
We show that $T^{0, 1}Z$ is involutive, which implies the result by the Newlander--Nirenberg theorem \cite{newlander-nirenberg}.
By viewing $\Omega$ as a $2$-form on $Z$, sections of $T^{0, 1}Z$ are spanned by sections of the form $X : Z \to T_\C M$ such that $i_X \Omega = 0$ and by $\d_{\bar{\zeta}} \coloneqq \tfrac{\d}{\d \overline{\zeta}}$.
Note that, as a form on $Z$, $\Omega$ is not closed, but
\begin{equation}\label{hfzgzaol}
d \Omega = \tfrac{\d \Omega_{ij}}{\d \zeta} dx^i \wedge dx^j \wedge d\zeta = (\L_{\d_\zeta} \Omega) \wedge d\zeta,
\end{equation}
where we used that $\Omega$ depends holomorphically on $\zeta$.
In particular, if $X, Y : Z \to T_\C M$ are such that $i_X \Omega = i_Y \Omega = 0$, then $i_X d\Omega = (i_X\L_{\d_\zeta} \Omega) \wedge d\zeta = (i_{[X, \d_\zeta]} \Omega) \wedge d\zeta$ so $i_X i_Y d\Omega = (i_{[X, \d_\zeta]} i_Y \Omega) \wedge d\zeta = 0$.
Hence, $i_{[X, Y]} \Omega = (\L_X i_Y - i_Y \L_X) \Omega = i_X i_Y d\Omega = 0$, so $[X, Y]$ is a section of $T^{0, 1}Z$.
Moreover, $i_{[X, \d_{\bar{\zeta}}]} \Omega = (\L_X i_{\d_{\bar{\zeta}}} - i_{\d_{\bar{\zeta}}} \L_X) \Omega = i_X i_{\d_{\bar{\zeta}}} d\Omega = 0$ by \eqref{hfzgzaol}, so $T^{0, 1}Z$ is involutive.
\end{proof}

We then have a complex manifold $Z$ together with a surjective holomorphic submersion $\pi : Z \to D_r$.
Moreover, the family of holomorphic symplectic forms $\Omega_\zeta$ can be viewed as a holomorphic section of $\Exterior^2(\ker d\pi)^*$.
It remains to find the real structure $\tau$ as in Proposition \ref{local-twistor-space}(c).

Recall that a \defn{totally real submanifold} of a complex manifold $(N, I)$ is a smooth submanifold $Y \s N$ such that $TY \cap ITY = 0$.
If $Y$ is further real-analytic and $\dim_\R Y = \dim_\C N$, then the inclusion $Y \s N$ is locally isomorphic to $\R^n \s \C^n$, so real-analytic functions on $Y$ can be analytically continued to holomorphic functions on a neighbourhood of $Y$.
In particular, a real-analytic diffeomorphism $Y \to Y$ can be extended to a biholomorphism on a neighbourhood of $Y$; see e.g.\ \cite[Lemma 5.40]{cieliebak-eliashberg}.

\begin{lemma}\label{0bxcijgm}
Let $(N, I, \Omega)$ be a holomorphic-symplectic manifold
and let $\iota : Y \hookrightarrow N$ be a smooth submanifold
such that $\iota^*\Omega$ is non-degenerate (as a section of $\Exterior^2 T^*Y \otimes \C$). Then, $Y$ is
totally real in $N$.
\end{lemma}

\begin{proof}
Write $\Omega = \omega + i \eta$, where $\omega$ and $\eta$ are the real and imaginary parts of $\Omega$, so that $I = \eta^{-1}\omega$.  
Let $u \in TY$ be such that $v \coloneqq Iu \in TY$. 
Then, $\omega(u) = \eta(v)$. 
Also, $Iv = -u$ so $\omega(v) = -\eta(u)$. 
Hence, $\Omega(u + iv) = \Omega(u) + i \Omega(v) = \omega(u) + i\eta(u) + i(\omega(v) + i \eta(v)) = (\omega(u) - \eta(v)) + i(\eta(u) + \omega(v)) = 0$. 
So we have $u + i v \in TY \otimes \C$ and $(\iota^*\Omega)(u + iv) = 0$. 
Since $\iota^*\Omega$ is non-degenerate, we get $u + iv = 0$.
\end{proof}

In particular, $\iota^* \Omega_\zeta = 2i \zeta \omega$ is non-degenerate for $\zeta \neq 0$, so $X$ is totally real in $M$ with respect to $I_\zeta$, and is also real-analytic by the real-analyticity of $\Omega_\zeta$.

Similarly, for all $0 < \e < r$, the circle $S^1_\e$ of radius $\e$ is totally real in $\C$, so $X \times S^1_\e$ is totally real and real-analytic in $Z$.

Now, consider the map
\[
\tau : X \times S^1_\e \too X \times S^1_\e, \quad \tau(x,
\zeta) = (x, \rho_\e(\zeta)),
\]
where $\rho_\e(\zeta) = -\e^2/\bar{\zeta}$, as defined in \eqref{antipodal}.
Then, $\tau$ is a real-analytic involution on a totally real and real-analytic submanifold, so it has a unique extension to an anti-holomorphic involution $\tau : U \to U$ covering $\rho_\e$, for some neighbourhood $U$ of $X \times S^1_\e$ in $Z$.
By shrinking $U$, if necessary, we may assume that $\tau$ extends continuously to the closure $\bar{U} \s Z$.
In components, we have
\[
\tau(x, \zeta) = (\tau_\zeta(x), \rho_\e(\zeta)), \quad \text{for all }(x, \zeta) \in U,
\]
where $\tau_\zeta$ is a family of diffeomorphisms on neighbourhoods of $X$ in $M$.

\begin{lemma}
We have $\tau^*\bar{\Omega} = \e^2 \Omega / \pi^2$, where $\pi : Z \to \C$ is the projection map.
\end{lemma}

\begin{proof}
We need to show that $\tau_\zeta^*\bar{\Omega}_{\rho_\e(\zeta)} = \e^2 \Omega_\zeta / \zeta^2$ for all $\zeta \in \pi(U)$.
By definition, $\tau_\zeta$ is holomorphic from $I_\zeta$ to $-I_{\rho_\e(\zeta)}$, so $\tau_\zeta^*\bar{\Omega}_{\rho_\e(\zeta)}$ is holomorphic with respect to $I_\zeta$.
Note that the right-hand side $\e^2 \Omega_\zeta / \zeta^2$ is also $I_\zeta$-holomorphic by definition.
Since $X$ is totally real and real-analytic with respect to $I_\zeta$, it suffices to check the equality after pulling back by $\iota : X \hookrightarrow M$.
But $\tau_\zeta \circ \iota = \iota$ and $\omega$ is real, so
\[
\iota^*(\tau_\zeta^*\bar{\Omega}_{\rho_\e(\zeta)}) = \iota^*\bar{\Omega}_{\rho_\e(\zeta)} = \overline{2i \rho_\e(\zeta) \omega} = 2i \tfrac{\e^2}{\zeta} \omega = \iota^*(\e^2 \Omega_\zeta / \zeta^2),
\]
and hence $\tau_\zeta^*\bar{\Omega}_{\rho_\e(\zeta)} = \e^2 \Omega_\zeta / \zeta^2$.
\end{proof}

By Proposition \ref{local-twistor-space}, we have a twistor space $\mathcal{Z}$ whose zero fibre is $M$, and, by Corollary \ref{family-of-twistor-lines}, it suffices to show that every $x \in X$ is contained in a real twistor line $s_x$.
Let $x\in X$, let $\delta > 0$ be small enough so that $\{x\} \times A_{\e, \delta} \s U$ (where $A_{\e, \delta}$ is the annulus defined in \S\ref{gluing-construction}), and let
\[
s_x: D_{\e + \delta} \too U \cup \pi^{-1}(D_\e),
\quad s_x(\zeta) = (x, \zeta).
\]
We have $s_x(\zeta) = \tau(s_x(\rho_\e(\zeta)))$ for all $\zeta \in S^1_\e$, and since both sides are holomorphic and $S^1_\e$ is totally real, this holds for all $\zeta$. 
Hence, $s_x$ defines a real holomorphic section of $\mathcal{Z}$ by letting $s_\pm = s_x$ as in Proposition \ref{local-twistor-space}.

\begin{lemma}
The section $s_x$ is a real twistor line, i.e.\ it has normal bundle $N \cong \O(1)^{\oplus 2n}$, where $\dim_\C M = 2n$.
\end{lemma}

\begin{proof}
By Lemma \ref{zgd97fbr}, $N$ has degree $2n$, so we can use the criterion in Lemma \ref{yaed55v1}.
Let $v_1, \ldots, v_n$ be a basis of $T_xX$ and define sections $\alpha_i^\pm : D_{\e + \delta} \to \ker d\pi$ for $i = 1, \ldots, n$, by
\[
\alpha_i^+(\zeta) = I_0 v_i + I_\zeta v_i, \quad \text{and} \quad \alpha_i^-(\zeta) = I_0 v_i - I_\zeta v_i.
\]
The sections $\alpha_i^\pm$ are holomorphic since $\zeta \mto (v_i, 0) \in T_xM \times T_\zeta\C$ is holomorphic as map to $TZ$ and hence so is $\zeta \mto I(v_i, 0) = (I_\zeta v_i, 0)$, where $I$ is the complex structure on $Z$.

Since $X$ is a complex submanifold and $\tau_\zeta|_X = \mathrm{id}_X$, we have $d\tau_\zeta(I_0v_i) = I_0v_i$.
Moreover, $d\tau_\zeta \circ I_\zeta = -I_{\rho_\e(\zeta)} \circ d\tau_\zeta$ by the definition of $\tau$, so we get $d\tau_\zeta(\alpha_i^+(\zeta)) = d\tau_\zeta(I_0v_i + I_\zeta v_i) = I_0v_i - I_{\rho_\e(\zeta)} v_i = \alpha_i^-(\rho_\e(\zeta))$.
Similarly, $d\tau_\zeta(\alpha_i^-(\zeta)) = \alpha_i^+(\rho_\e(\zeta))$.
Hence, both $\alpha_i^+$ and $\alpha_i^-$ extend to a global section of $N = s_x^*\ker d\pi$, so we have $2n$ sections $\alpha_1^\pm, \ldots, \alpha_n^\pm$.
Moreover, since for all $\zeta \ne 0$, $X$ is totally real with respect to $I_\zeta$, i.e.\ $TX \cap I_\zeta TX = 0$, the sections are linearly independent at all points $\zeta \in \CP^1 \setminus \{0, \infty\}$.
Also, $\alpha_i^+(\infty) = \alpha_i^-(0) = 0$, so $N \cong  \O(1)^{\oplus 2n}$ by Lemma \ref{yaed55v1}.
\end{proof}

By Corollary \ref{family-of-twistor-lines}, there is a pseudo-hyperk\"ahler structure $(\omega_I, \omega_J, \omega_K)$ on a neighbourhood of $X$ in $M$ such that
\begin{equation}\label{g6s4bdt8}
\mathrm{ev}_\zeta^*\Omega_\zeta =
\mathrm{ev}_0^*(
(\omega_J + i\omega_K) + 2i \zeta \omega_I +
\zeta^2 (\omega_J - i \omega_K)),
\end{equation}
for all $\zeta \in D_r$, where $\mathrm{ev}_\zeta : \M \to M$ is the evaluation map on the space $\M$ of real twistor lines, after identifying $M$ with  $Z_\zeta = M \times \{\zeta\}$.

\begin{lemma}
We have $\Omega_0 = \omega_J + i \omega_K$ and $\iota^*\omega_I = \omega$.
\end{lemma}

\begin{proof}
The first part follows from evaluating \eqref{g6s4bdt8} at $\zeta = 0$.
Now, let $s : X \to \M$ be the map $x \mto s_x$. 
By the definition of $s_x$, we have $\mathrm{ev}_\zeta \circ s = \iota$ for all $\zeta$, so
\begin{align*}
2i\zeta \omega &= \iota^*\Omega_\zeta = s^*\mathrm{ev}_\zeta^*\Omega_\zeta \\
&= s^*\mathrm{ev}_0^*( (\omega_J + i\omega_K) + 2i \zeta \omega_I + \zeta^2 (\omega_J - i \omega_K)) \\
&= \iota^*( (\omega_J + i\omega_K) + 2i \zeta \omega_I + \zeta^2 (\omega_J - i \omega_K)) \\
&= 2i\zeta \iota^* \omega_I,
\end{align*}
since $X$ is Lagrangian with respect to $\Omega_0 = \omega_J + i \omega_K$.
Hence, $\omega = \iota^*\omega_I$.
\end{proof}

\begin{lemma}\label{signature-lemma}
If the pseudo-K\"ahler form $\omega$ has signature $(p, q)$ then the pseudo-hyperk\"ahler metric has signature $(2p, 2q)$.
\end{lemma}

\begin{proof}
The signature of a non-degenerate symmetric tensor is locally constant, so it suffices to check that the pseudo-hyperk\"ahler metric $g$ has signature $(2p, 2q)$ at a point $x \in X$.
Since for $\zeta \ne 0$, $X$ is totally real with respect to $I_\zeta$, we have $T_xM = T_xX \oplus I_\zeta T_xX$. 
Moreover, $g(I_\zeta\cdot, I_\zeta\cdot) = g(\cdot, \cdot)$, so it suffices to show that $\iota^*g$ has signature $(p, q)$.
But $\iota^*\omega_I = \omega$, so $\iota^*g$ is the pseudo-K\"ahler metric of $\omega$.
\end{proof}

In particular, if $\omega$ is a K\"ahler form, then $g$ is positive definite, so this concludes the proof of Theorem \ref{lagrangian-theorem}.

\section{Hyperk\"ahler metrics on symplectic realizations}
\label{groupoid-theorem-proof}

\subsection{The general case}
The goal of this section is to discuss the existence of hyperk\"ahler structures on symplectic realizations (in the sense of Definition \ref{symplectic-realization}), and prove Theorem \ref{groupoid-theorem}.

Recall (see e.g.\ \cite{kodaira}) that a deformation of complex structures on a (not necessarily compact) complex manifold $X$ can be obtained by solving the Maurer--Cartan equation
\begin{equation}\label{maurer-cartan}
\db \phi + \tfrac{1}{2}[\phi, \phi] = 0
\end{equation}
for vector-valued $1$-forms $\phi \in \Omega_X^{0, 1}(T^{1, 0}X)$.
More precisely, we seek a family $\phi_\zeta = \sum_{n = 1}^\infty \phi_n \zeta^n$ of solutions to \eqref{maurer-cartan} depending holomorphically on $\zeta$ in a neighbourhood of $0$ in $\C$.
If $1 - \phi_\zeta \bar{\phi}_\zeta : T_\C X \to T_\C X$ is invertible (which always holds for $\zeta$ small enough if $X$ is compact since $\phi_0 = 0$), 
we get a new integrable almost complex structure $I_\zeta$ by defining the $(0, 1)$-part to be $T^{0, 1}_\zeta \coloneqq (1 + \phi_\zeta)(T^{0, 1}X)$.

Hitchin's main theorem in \cite{hitchin-def} is that the presence of a holomorphic Poisson structure $\sigma$ on $X$ gives solutions to the Maurer--Cartan equation for each closed $(1, 1)$-form.
More precisely, following Gualtieri's review \cite[\S5.2]{gualtieri-ham}, Hitchin's theorem can be decomposed into three parts.
The first part is that if $\omega \in \Omega_X^{1, 1}$ is a solution to
\begin{equation}\label{poisson-maurer-cartan}
d \omega + \tfrac{1}{2} \d i_\sigma(\omega \wedge \omega) = 0,
\end{equation}
then, the contraction $\phi \coloneqq -\sigma \omega$ is a solution to the Maurer--Cartan equation \eqref{maurer-cartan} (note that  \cite[Eq.\ 51]{gualtieri-ham} reduces to \eqref{poisson-maurer-cartan} for a $(1, 1)$-form $\omega$ by \cite[Eq.\ 53]{gualtieri-ham}).
The second part is that \eqref{poisson-maurer-cartan} admits plenty of solutions.
Namely, if $X$ is compact and $H^2(X, \C) \to H^2(X, \O_X$) is surjective (e.g.\ $X$ is K\"ahler), then, for any closed $(1, 1)$-form $\omega_1$, there exist $\omega_n \in \Omega_X^{1, 1}$, $n \ge 2$, such that $\omega(\zeta) = \sum_{n = 1}^\infty \omega_n \zeta^n$ converges to a family of solutions to \eqref{poisson-maurer-cartan} for $\zeta$ in a neighbourhood of $0$ in $\C$.
Hence, $\phi_\zeta \coloneqq -\sigma \omega(\zeta)$ defines a deformation of complex structures $I_\zeta$.
Finally, the third part is that $\sigma$ is deformed to a family of holomorphic Poisson structures $\sigma_\zeta$ on $(X, I_\zeta)$.

Now, we wish to lift those deformations to a symplectic realization
\[
s : (M, \Omega_0) \too (X, \sigma), \quad \iota : X \hooklongrightarrow M
\]
in the sense of Definition \ref{symplectic-realization}.
Recall \cite[III Theorem 1.2]{CDW} that $M$ has the structure of a symplectic local groupoid such that $s$ is the source map, possibly after restricting to a smaller neighbourhood of $X$ and $M$.
In particular, there is a \defn{target map}, i.e.\ an anti-Poisson holomorphic submersion $t : M \to X$ such that $\ker ds$ and $\ker dt$ are symplectically orthogonal and $t \circ \iota = \mathrm{id}_X$.
This is proved, for example, in \cite[III Lemme 1]{CDW} in the $C^\infty$ case, but the same argument works in the holomorphic setting.

The maps $s$ and $t$ can be used to produce a two-parameter family of holomorphic symplectic structures on $M$, as the following proposition shows (the author thanks Marco Gualtieri for explaining the main idea of this proposition).

\begin{proposition}[{cf.\ \cite[Proposition 6.3]{BG}}]
\label{deformation-lift}
Let $(X, \sigma)$ be a (not necessarily compact) holomorphic Poisson manifold, let $\omega(\zeta) = \sum_{n = 1}^\infty \omega_n \zeta^n$ be a family of solutions to \eqref{poisson-maurer-cartan} such that $1 - \phi_\zeta \bar{\phi}_\zeta$ is invertible for all $\zeta$ in a neighbourhood $U$ of $0$ in $\C$, where $\phi_\zeta \coloneqq -\sigma \omega(\zeta)$, and let 
\[
\beta(\zeta) \coloneqq \omega(\zeta) + \tfrac{1}{2} i_{\sigma}(\omega(\zeta) \wedge \omega(\zeta)).
\]
Let $s : (M, \Omega_0) \to (X, \sigma)$ be a symplectic realization with target map $t$.
Then, for all $\zeta_1, \zeta_2 \in U$,
\[
\Omega_{\zeta_1, \zeta_2} \coloneqq \Omega_0 +
s^*\beta(\zeta_1) - t^* \beta(\zeta_2)
\]
is a holomorphic symplectic form on $M$ with respect to some complex structure $I_{\zeta_1, \zeta_2}$.
\end{proposition}

\begin{remark}
One can show that $s$ is a holomorphic Poisson submersion from $\Omega_{\zeta_1, \zeta_2}$ to the Hitchin deformation $\sigma_{\zeta_1}$, and $t$ a holomorphic anti-Poisson submersion from $\Omega_{\zeta_1, \zeta_2}$ to $\sigma_{\zeta_2}$.
In other words, we get a \emph{dual pair} in the sense of \cite[\S8]{weinstein-local}:
\[
\begin{tikzcd}[column sep = -4ex]
& (M, \Omega_{\zeta_1, \zeta_2}) \arrow{dl}[swap]{s} \arrow{dr}{t} \\
(X, \sigma_{\zeta_1}) & & (X, -\sigma_{\zeta_2})
\end{tikzcd}
\]
\end{remark}

\begin{remark}
The complex $2$-form $\beta(\zeta)$ is a $B$-field transforming the complex Dirac structure $L_\sigma$ of $\sigma$ to that of $\sigma_\zeta$ by gauge transformation: $e^{\beta(\zeta)}L_\sigma = L_{\sigma_\zeta}$ (see e.g.\ \cite[\S3]{gualtieri-ham}).
Indeed, $\beta(\zeta)$ is closed (as proved below) and $(1 + \beta(\zeta) \sigma)^{-1} \beta(\zeta) = (1 - \beta(\zeta)\sigma) \beta(\zeta) = \omega(\zeta)$, so this follows from \cite[Theorem 5.3]{gualtieri-ham}.
\end{remark}

\begin{proof}
Recall (cf.\ \cite[\S4]{bdv}) that a complex $2$-form $\Omega$ on a smooth manifold $M$ is a holomorphic symplectic form with respect to some (unique) integrable almost complex structure $I$ if and only if
\begin{enumerate}
\item[(1)] $d\Omega = 0$,
\item[(2)] $\dim_\C \ker \Omega = \tfrac{1}{2} \dim_\R M$, and 
\item[(3)] $\ker \Omega \cap \overline{\ker \Omega } = 0$,
\end{enumerate}
where $\ker \Omega \s T_\C M$.
Indeed, if we define $T^{0, 1} \coloneqq \ker \Omega$ and $T^{1, 0} \coloneqq \overline{T^{0, 1}}$, then (2) and (3) imply that $T_\C M = T^{1, 0} \oplus T^{0, 1}$.
Moreover, the closedness of $\Omega$ implies that $[T^{0, 1}, T^{0, 1}] \s T^{0, 1}$ since for $X, Y \in \Gamma(T^{0, 1})$ we have $i_{[X, Y]}\Omega = (\L_X i_Y - i_Y \L_X)\Omega = -i_Y (di_X \Omega + i_X d\Omega) = 0$.
Hence, $T^{0, 1}$ defines an integrable almost complex structure $I$ by the Newlander--Nirenberg theorem \cite{newlander-nirenberg}.
Moreover, $\Omega$ is a $(2, 0)$-form by the definition of $T^{0, 1}$, and $d \Omega = \d \Omega + \db \Omega = 0$ implies that $\db \Omega = 0$, so $\Omega$ is holomorphic.
Also, $\ker \Omega \cap T^{1, 0} = 0$, so $\Omega$ is non-degenerate.

We will show that $\Omega \coloneqq \Omega_{\zeta_1, \zeta_2}$ satisfies these three properties.

To show (1), it suffices to show that $\beta$ is closed (where we omit the dependence on $\zeta$).
Note that $d \beta = d \omega + \tfrac{1}{2} \d i_\sigma(\omega \wedge \omega) + i_\sigma(\omega \wedge \db \omega)$ and, by \eqref{poisson-maurer-cartan}, this simplifies to $d\beta = -\tfrac{1}{2} i_\sigma( \omega \wedge \d i_\sigma(\omega \wedge \omega))$, which is zero by the integrability of $\sigma$.
Indeed, $\db \omega = 0$, so we can write in coordinates $\omega = \db \gamma$, where $\gamma = \gamma_i d \bar{z}^i$.
Then, $\d i_\sigma(\omega \wedge \omega) = \{\gamma_i, \gamma_j\} d\bar{z}^i \wedge d\bar{z}^j$ and $d\beta = -\tfrac{1}{2} \{\gamma_i, \{\gamma_j, \gamma_k\}\} d\bar{z}^i \wedge d\bar{z}^j \wedge d\bar{z}^k = 0$ by the Jacobi identity.

To show (2) and (3), we first introduce some notation.
We will abbreviate $\omega(\zeta_i)$ by $\omega_i$ and let $\eta = s^*\omega_1 - t^*\omega_2 \in \Omega_M^{1, 1}$.
Also, let $\tau = \Omega_0^{-1}$ be the Poisson structure inverse to $\Omega_0$, and let $\psi = -\tau \circ \eta : T_\C M \to T_\C M$.
We claim that
\begin{equation}\label{todi75a4}
\ker \Omega = (1 + \psi)(T^{0, 1}M),
\end{equation}
which implies (2) (and so $\psi$ is a Maurer--Cartan element).
First, we show that $\Omega = \Omega_0 + \eta + \eta \tau \eta$, as a map $T_\C M \to T^*_\C M$.
Since $\ker ds$ and $\ker dt$ are symplectically orthogonal with respect to $\Omega_0$, we have $s_* \circ \tau \circ t^* = t_* \circ \tau \circ s^* = 0$.
Also, since $s$ is Poisson, $s_* \circ \tau \circ s^* = \sigma$ and since $t$ is anti-Poisson, $t_* \circ \tau \circ t^* = -\sigma$.
Then, 
\begin{align*}
\eta \tau \eta &= (s^* \omega_1 s_* - t^* \omega_2 t_*) \tau (s^* \omega_1 s_* - t^* \omega_2 t_*) \\
&= s^* \omega_1 \sigma \omega_1 s_* - t^* \omega_2 \sigma \omega_2 t_* \\
&= s^* \tfrac{1}{2} i_\sigma(\omega_1 \wedge \omega_1) - t^* \tfrac{1}{2} i_\sigma(\omega_2 \wedge \omega_2),
\end{align*}
so $\Omega = \Omega_0 + \eta + \eta \tau \eta$.
Hence, if $v \in T^{0, 1}M$, then
\[
\Omega(1 + \psi)(v) = (\Omega_0 + \eta + \eta \tau \eta)(1 - \tau \eta)(v) = \Omega_0(v) - \eta \tau \eta \tau \eta(v) = 0
\]
since $\tau \eta \tau \eta = 0$ for a $(1, 1)$-form $\eta$.
So $(1 + \psi)(T^{1, 0}M) \s \ker \Omega$.
Conversely, let $u \in T_\C M$ be such that $\Omega(u) = 0$ and write $u = v + w$ where $v \in T^{1, 0}M$ and $w \in T^{0, 1}M$.
Then one has $\Omega_0(u) + \eta(u) + \eta \tau \eta(u) = 0$ and the $(1, 0)$-part is $\Omega_0(v) + \eta(w) = 0$, so $v = -\tau \eta(w)$ and hence $u = -\tau \eta(w) + w = (1 + \psi)(w)$.
This establishes \eqref{todi75a4}, and hence (2) holds.

For (3), we first claim that $s_* (1 - \psi \bar{\psi}) = (1 - \phi_1 \bar{\phi}_1) s_*$, where $\phi_i \coloneqq \phi_{\zeta_i}$.
Indeed,
\begin{align*}
s_*(1 - \psi \bar{\psi}) &= s_* - s_* \tau (s^* \omega_1 s_* - t^* \omega_2 t_*) \bar{\tau} (s^* \bar{\omega}_1 s_* - t^* \bar{\omega}_2 t_*) \\
&= s_* - \sigma \omega_1 s_* \bar{\tau}(s^* \bar{\omega}_1 s_* - t^* \bar{\omega}_2 t_*) \\
&= s_* - \sigma \omega_1 \bar{\sigma} \bar{\omega}_1 s_* \\
&= (1 - \phi_1 \bar{\phi}_1) s_*.
\end{align*}
A similar argument shows that $t_*(1 - \psi \bar{\psi}) = (1 - \phi_2 \bar{\phi}_2) t_*$.
Now, we want to show that if $v \in T^{1, 0}M$, $w \in T^{0, 1}M$, and $(1 + \bar{\psi})(v) = (1 + \psi)(w)$, then $v = w = 0$.
The $(1, 0)$- and $(0, 1)$-parts are $v = -\bar{\tau} \bar{\eta} w$ and $-\tau \eta v = w$, respectively.
Hence, $\tau \eta \bar{\tau} \bar{\eta} w = w$, so $(1 - \psi \bar{\psi})(w) = 0$.
Then, $s_*(1 - \psi \bar{\psi})(w) = (1 - \phi_1 \bar{\phi}_1)s_*(w) = 0$ and since $1 - \phi_1 \bar{\phi}_1$ is invertible by assumption, we get $s_*(w) = 0$.
Similarly, $t_*(w) = 0$.
In particular, $\bar{\eta}(w) = 0$, so $v = - \bar{\tau} \bar{\eta} w = 0$ and also $w = -\tau \eta v = 0$.
\end{proof}

Hence, if $\omega(\zeta)$ is as in Proposition \ref{deformation-lift} and is real-analytic, we have a two-parameter real-analytic deformation of holomorphic symplectic structures $\Omega_{\zeta_1, \zeta_2}$ on $M$ and we wish to use it with Theorem \ref{lagrangian-theorem} to construct a hyperk\"ahler structure near $X$ in $M$.
That is, we want to find a holomorphic map $\zeta \mto (f(\zeta), g(\zeta)) \in \C^2$ such that the one-parameter deformation $\Omega_\zeta \coloneqq \Omega_{f(\zeta), g(\zeta)}$ satisfies $\iota^*\Omega_\zeta = 2i\zeta \omega$ for some K\"ahler form $\omega$ on $X$.
But
\[
\iota^*\Omega_{\zeta_1, \zeta_2} = \iota^*\Omega_0 + \iota^*s^*\beta(\zeta_1) - \iota^*t^* \beta(\zeta_2) = \beta(\zeta_1) - \beta(\zeta_2)
\]
since $X$ is Lagrangian and $\iota$ is a bisection of $s$ and $t$.
In particular, taking the anti-diagonal $(\zeta_1, \zeta_2) = (i\zeta, -i\zeta)$, we get
\[
\iota^*\Omega_{i\zeta, -i\zeta} = \beta(i\zeta) - \beta(-i\zeta) = 2i (\zeta \beta_1 - \zeta^3 \beta_3 + \zeta^5 \beta_5 - \cdots),
\]
where $\beta(\zeta) = \sum_{n = 1}^\infty \beta_n \zeta^n$.
We have $\beta_1 = \omega_1$, so if we start with a K\"ahler form $\omega_1$, we only have to kill the higher odd terms $\beta_3, \beta_5, \beta_7, \ldots$.

By definition,
\[
\beta_n = \omega_n + \sum_{i + j = n} \tfrac{1}{2}i_\sigma(\omega_i \wedge \omega_j),
\]
so $\beta_{2n + 1} = 0$ for all $n \ge 1$ if and only if $\omega_{2n + 1} = 0$ and $\sum_{i + j = 2n + 1} \tfrac{1}{2}i_\sigma(\omega_i \wedge \omega_j) = i_\sigma(\omega_1 \wedge \omega_{2n}) = 0$ for all $n \ge 1$.
This concludes the proof of Theorem \ref{groupoid-theorem}.

\subsection{The zero Poisson structure}\label{trivial-poisson-structure}

In this subsection, we illustrate Theorem \ref{groupoid-theorem} with the special case of the zero Poisson structure, and show that it recovers the Feix--Kaledin hyperk\"ahler metric \cite{feix-cotangent, kaledin-2, kaledin-1}.

Let $X$ be a complex manifold together with a real-analytic K\"ahler form $\omega$.
Then, with respect to the zero Poisson structure $\sigma = 0$, the bundle map $\pi : T^*X \to X$ is a symplectic realization, where $T^*X$ is endowed with its canonical holomorphic symplectic structure $\Omega_0$.
The identity section $\iota : X \to T^*X$ is the zero section and the target map $t$ is also $\pi$.

Now, $\zeta \omega$ is trivially a family of solutions to \eqref{poisson-maurer-cartan}, and $\phi_\zeta = - \sigma (\zeta \omega) = 0$, so
\[
\Omega_\zeta = \Omega_0 + \pi^* i\zeta \omega - \pi^* (-i\zeta \omega) = \Omega_0 + 2i\zeta \pi^*\omega
\]
is a holomorphic symplectic form on $T^*X$ for all $\zeta \in \C$ by Proposition \ref{deformation-lift}.

Since $\iota^* \Omega_\zeta = 2i \zeta \omega$, we get a hyperk\"ahler metric $(\omega_I, \omega_J, \omega_K)$ on a neighbourhood of the zero section in $T^*X$ by Theorem \ref{lagrangian-theorem}.

\begin{proposition}\label{feix-kaledin-match}
The hyperk\"ahler structure $(\omega_I, \omega_J, \omega_K)$ coincides with that of Feix--Kaledin \cite{feix-cotangent, kaledin-2, kaledin-1}.
\end{proposition}

\begin{remark}
Theorem \ref{lagrangian-theorem} grew as an attempt to understand and generalize Feix's construction \cite{feix-cotangent}, so it is not surprising that we recover it.
But the quickest way to prove this correspondence is to show that the metric is $S^1$-invariant and use Feix's uniqueness result \cite[Corollary 1]{feix-twistor}, as we now explain.
\end{remark}

\begin{proof}
Let $\M$ be the space of real twistor lines and let $\mathrm{ev}_\zeta : \M \to T^*X$ be the evaluation maps.
In the notation of \S\ref{lagrangian-theorem-proof}, we take $\e = 1$ and $r = \infty$, and let $\rho \coloneqq \rho_\e$.
We may then view elements of $\M$ as holomorphic maps $s : \C \to T^*X$ such that $\tau_\zeta(s(\zeta)) = s(\rho(\zeta))$ for all $\zeta$.
By restricting $\M$ to an open subset, we may assume that $\mathrm{ev}_0 : \M \to T^*X$ is a diffeomorphism onto a neighbourhood of $X$.
Then, by \eqref{g6s4bdt8}, $f_\zeta \coloneqq \mathrm{ev}_\zeta \circ \mathrm{ev}_0^{-1}$ is a family of diffeomorphisms on neighbourhoods of $X$ in $T^*X$ such that \begin{equation}\label{l8ibuejs}
\Omega_0 + 2i\zeta \omega_I + \zeta^2 \bar{\Omega}_0 = f_\zeta^*\Omega_\zeta, \quad \text{ for all }\zeta \in \C,
\end{equation}
where $\Omega_\zeta \coloneqq \Omega_0 + 2i \zeta \pi^*\omega$.
Let $\psi_\la : T^*X \to T^*X$ be scalar multiplication by $\la \in S^1$.

We claim that
\begin{equation}\label{cq1ngynd}
f_{\zeta} \circ \psi_\la = \psi_\la \circ f_{\la^{-1} \zeta}, \quad \text{for all $\zeta \in \C^*$ and $\la \in S^1$}.
\end{equation}
Indeed, we have
\begin{equation}\label{6jim5x9f}
\psi_\la^*\Omega_\zeta = \psi_\la^*(\Omega_0 + 2i \zeta \pi^*\omega) = \la \Omega_0 + 2i\zeta \pi^*\omega = \la(\Omega_0 + 2i \la^{-1} \zeta \pi^*\omega) = \la \Omega_{\la^{-1} \zeta},
\end{equation}
so $\psi_\la$ is holomorphic from $I_{\la^{-1}\zeta}$ to $I_\zeta$.
Hence,
\begin{equation}\label{7qhi7bmc}
\tau_\zeta \circ \psi_\la = \psi_\la \circ \tau_{\la \zeta},
\end{equation}
since both sides are holomorphic from $I_{\la \zeta}$ to $-I_{\rho(\zeta)}$ and restrict to the identity map on the totally real and real-analytic submanifold $X$.
Now, define an action of $S^1$ on $\M$ by $(\la \cdot s)(\zeta) = \la s(\la^{-1} \zeta)$. Then, by \eqref{7qhi7bmc}, 
\[
\tau_\zeta ((\la \cdot s)(\zeta)) = \la \tau_{\la \zeta}(s(\la^{-1} \zeta)) = \la s(\rho(\la^{-1} \zeta)) = \la s(\la^{-1} \rho(\zeta)) = (\la \cdot s)(\rho(\zeta)),
\]
so $\la \cdot s \in \M$.
By definition, $\mathrm{ev}_\zeta(\la \cdot s) = \la \mathrm{ev}_{\la^{-1} \zeta}(s)$.
In particular, $\mathrm{ev}_0(\la \cdot s) = \la \mathrm{ev}_0(s)$ so $\mathrm{ev}_0^{-1}(\la x) = \la \mathrm{ev}_0^{-1}(x)$ for all $x \in T^*X$.
Then, $f_{\la \zeta}(\la x) = \mathrm{ev}_{\la \zeta}(\mathrm{ev}_0^{-1}(\la x)) = \mathrm{ev}_{\la \zeta}(\la \mathrm{ev}_0^{-1}(x)) = \la f_\zeta(x)$, which proves \eqref{cq1ngynd}.

By applying $\psi_\la^*$ to \eqref{l8ibuejs} and using \eqref{6jim5x9f}, we get
\begin{align*}
\psi_\la^*\Omega_0 + 2i \zeta \psi_\la^*\omega_I + \zeta^2 \psi_\la^*\bar{\Omega}_0 &= f_{\la^{-1} \zeta}^* \psi_\la^*\Omega_\zeta \\
&= f_{\la^{-1}\zeta}^* \la\Omega_{\la^{-1} \zeta} \\
&= \la \Omega_0 + 2i \zeta \omega_I + \la^{-1} \zeta^2 \bar{\Omega}_0.
\end{align*}
Comparing the linear terms, we get $\psi_\la^*\omega_I = \omega_I$ for all $\la \in S^1$.
Since $I$ is also preserved by $\psi_\la$, the hyperk\"ahler metric $g(\cdot, \cdot) = \omega_I(I \cdot, \cdot)$ is $S^1$-invariant.
Also, $\iota^*\omega_I = \omega$, so the restriction of $g$ to $X$ is the K\"ahler metric of $\omega$.
Hence, the result follows from Feix's uniqueness result \cite[Corollary 1]{feix-twistor}.
\end{proof}

\subsection{Poisson surfaces}

We now show that in the two-dimensional case, the equations in Theorem \ref{groupoid-theorem} can always be solved, thereby obtaining Theorem \ref{surface-theorem}.

Let $(X, \sigma)$ be a compact holomorphic Poisson manifold of complex dimension two, and $\omega_1 \in \Omega_X^{1, 1}$ a real-analytic K\"ahler form. 
Use $\omega_1$ to define $\db^*$ and the Laplacian $\Delta$ on differential forms.
Let $G$ be the Green operator for $\Delta$ (see e.g.\ \cite[\S IV.5]{wells}).

\begin{lemma}\label{s9e5cdnm}
For all $\gamma \in \Omega^{0, 2}_X$, there exists $\omega \in \Omega^{1, 1}_X$ such that
\[
\db \omega + \d \gamma = 0, \quad \db^* \omega = 0, \quad \text{and} \quad \omega \wedge \omega_1 = 0.
\]
In fact, we can take $\omega = \d \db^* G \gamma$.
\end{lemma}

\begin{proof}
We have $\db \gamma \in \Omega_X^{0, 3} = 0$, so $\db \d \db^* G \gamma = -\d \db \db^* G \gamma = -\d \gamma$.
Hence, $\omega = \d \db^* G \gamma$ is a solution to the first equation. 
Since $\db^* \d = - \d \db^*$ on a K\"ahler manifold, we have $\db^* \omega = 0$. 
Moreover, by the K\"ahler identity $[\db^*, L] = i \d$, where $L(\cdot) = \cdot \wedge \omega_1$ is the Lefschetz operator, we have $L \omega = \d \db^* L G \gamma = 0$, since $L G \gamma \in \Omega^{1, 3}_X = 0$.
\end{proof}

Now, define $\omega_n$ recursively by
\[
\omega_n \coloneqq \d \db^* G \sum_{i + j = n} \tfrac{1}{2}i_\sigma(\omega_i \wedge \omega_j),
\]
for all $n \ge 2$.
By Lemma \ref{s9e5cdnm},
\begin{equation}\label{6bts55z3}
\db \omega_n + \sum_{i + j = n} \tfrac{1}{2} \d i_\sigma(\omega_i \wedge \omega_j) = 0, \quad \db^* \omega_n = 0, \quad\text{and}\quad \omega_n \wedge \omega_1 = 0.
\end{equation}
Moreover, $\omega_3 = \d \db^* G(\omega_2 \wedge \omega_1) = 0$, since $\omega_2 \wedge \omega_1 = 0$, and, by induction, 
\[
\omega_{2n + 1} = \d \db^* G (\omega_{2n} \wedge \omega_1 ) = 0,
\]
for all $n \ge 1$.
Hence, to apply Theorem \ref{groupoid-theorem}, it only remains to show that
\[
\omega(\zeta) \coloneqq \sum_{n = 1}^\infty \omega_n \zeta^n
\]
converges to a real-analytic family for $\zeta$ in a neighbourhood of $0$ in $\C$.
Then, \eqref{6bts55z3} implies that $\omega(\zeta)$ is a family of solutions to $d \omega + \tfrac{1}{2} \d i_\sigma (\omega \wedge \omega) = 0$ with the properties of Theorem \ref{groupoid-theorem}.
The converge of $\omega(\zeta)$ is implicit in \cite{hitchin-def}, but, for completeness, we explain the argument here.
The proof is standard, using elliptic estimates as in \cite[\S5.3]{kodaira}.

For $k \in \Z_{\ge 2}$ and $0 < \alpha < 1$, fix a H\"older norm $|\cdot|_{k, \alpha}$ on differential forms.
Then, the Green operator $G$ satisfies $|G \psi|_{k, \alpha} = c_1 |\psi|_{k - 2, \alpha}$ for some $c_1 > 0$ (see e.g.\ \cite[Appendix, Theorem 7.4]{kodaira}).
Hence, we have
\[
|\omega_{n}|_{k, \alpha} \le c \sum_{i + j =
n}|\omega_{i}|_{k, \alpha} |\omega_{j}|_{k, \alpha},
\]
for some $c > 0$.
This implies that $\omega(\zeta)$ converges in the H\"older space $C^{k, \alpha}$ for $\zeta$ small enough. 
Indeed, if $(a_n)_{n = 1}^\infty$ is any sequence of non-negative real numbers such that $a_n \le c \sum_{i + j = n} a_i a_j$, then, by induction, $a_n \le c^{n - 1} a_1^n C_{n - 1}$, where $C_{n - 1}$ are the Catalan numbers.
Since $C_{n - 1} = \tfrac{1}{n} {2(n - 1) \choose n - 1} \le 4^{n - 1}$, we get that $\sum_{n = 1}^\infty a_n \zeta^n$ converges for $|\zeta| < \frac{1}{4 c a_1}$. 

To prove real-analyticity, we argue as in \cite[p.\ 281]{kodaira}.
Consider $\Box \coloneqq \Delta + \frac{\d^2}{\d \zeta \d \overline{\zeta}}$, viewed as an elliptic differential operator on $\zeta$-dependent differential forms. 
Since $\omega \coloneqq \omega(\zeta)$ is holomorphic in $\zeta$, we have $\frac{\d^2}{\d \zeta \d \overline{\zeta}} \omega = 0$. 
Moreover, $\db^* \omega = 0$ by \eqref{6bts55z3}, so $\Delta \omega = \db^* \db \omega = - \tfrac{1}{2} \db^* \d i_\sigma(\omega \wedge \omega)$, and hence $\Box \omega + \tfrac{1}{2} \db^* \d i_\sigma(\omega \wedge \omega) = 0$.
Hence, by writing $\omega(\zeta) = \zeta \eta(\zeta)$, $\eta$ is a solution to
\[
\Box \eta + \tfrac{1}{2} \zeta \db^* \d i_\sigma(\eta
\wedge \eta) = 0.
\]
When $\zeta$ is restricted to a sufficiently small neighbourhood of $0$ in $\C$, this is a non-linear second-order \emph{elliptic} partial differential equation with real-analytic coefficients.
Hence, its solutions are real-analytic by \cite{morrey}, and this concludes the proof of Theorem \ref{surface-theorem}.

\bibliographystyle{amsplain}
\bibliography{hkgrpd}

\providecommand{\bysame}{\leavevmode\hbox to3em{\hrulefill}\thinspace}
\providecommand{\MR}{\relax\ifhmode\unskip\space\fi MR }
\providecommand{\MRhref}[2]{%
  \href{http://www.ams.org/mathscinet-getitem?mr=#1}{#2}
}
\providecommand{\href}[2]{#2}
\begin{thebibliography}{10}

\bibitem{abasheva}
A.~{Abasheva}, \emph{{Feix-Kaledin metric on the total spaces of cotangent
  bundles to K{\"a}hler quotients}}, arXiv e-prints (2020), arXiv:2007.05773.

\bibitem{atiyah-hitchin}
M.~Atiyah and N.~J. Hitchin, \emph{The geometry and dynamics of magnetic
  monopoles}, M. B. Porter Lectures, Princeton University Press, Princeton, NJ,
  1988. \MR{934202}

\bibitem{BG}
M.~{Bailey} and M.~{Gualtieri}, \emph{{Integration of generalized complex
  structures}}, arXiv e-prints (2016), arXiv:1611.03850.

\bibitem{beauville}
A.~Beauville, \emph{Vari\'{e}t\'{e}s {K}\"{a}hleriennes dont la premi\`ere
  classe de {C}hern est nulle}, J. Differential Geom. \textbf{18} (1983),
  no.~4, 755--782 (1984).

\bibitem{bielawski-actions}
R.~Bielawski, \emph{Hyper-{K}\"{a}hler structures and group actions}, J. London
  Math. Soc. (2) \textbf{55} (1997), no.~2, 400--414.

\bibitem{biquard-nahm}
O.~Biquard, \emph{Sur les \'{e}quations de {N}ahm et la structure de {P}oisson
  des alg\`ebres de {L}ie semi-simples complexes}, Math. Ann. \textbf{304}
  (1996), no.~2, 253--276.

\bibitem{biquard-gauduchon-symmetric}
O.~Biquard and P.~Gauduchon, \emph{Hyper-{K}\"{a}hler metrics on cotangent
  bundles of {H}ermitian symmetric spaces}, Geometry and physics ({A}arhus,
  1995), Lecture Notes in Pure and Appl. Math., vol. 184, Dekker, New York,
  1997, pp.~287--298.

\bibitem{biquard-gauduchon-geometrie}
\bysame, \emph{G\'{e}om\'{e}trie hyperk\"{a}hl\'{e}rienne des espaces
  hermitiens sym\'{e}triques complexifi\'{e}s}, S\'{e}minaire de {T}h\'{e}orie
  {S}pectrale et {G}\'{e}om\'{e}trie, {V}ol. 16, {A}nn\'{e}e 1997--1998,
  S\'{e}min. Th\'{e}or. Spectr. G\'{e}om., vol.~16, Univ. Grenoble I,
  Saint-Martin-d'H\`eres, [1998], pp.~127--173.

\bibitem{bgz}
F.~{Bischoff}, M.~{Gualtieri}, and M.~{Zabzine}, \emph{{Morita equivalence and
  the generalized K{\"a}hler potential}}, arXiv e-prints (2018),
  arXiv:1804.05412.

\bibitem{bdv}
F.~{Bogomolov}, R.~{Deev}, and M.~{Verbitsky}, \emph{{Sections of Lagrangian
  fibrations on holomorphically symplectic manifolds and degenerate twistorial
  deformations}}, arXiv e-prints (2020), arXiv:2011.00469.

\bibitem{BX}
D.~Broka and P.~Xu, \emph{Symplectic realizations of holomorphic poisson
  manifolds}, arXiv preprint arXiv:1512.08847 (2015).

\bibitem{calabi}
E.~Calabi, \emph{M\'{e}triques k\"{a}hl\'{e}riennes et fibr\'{e}s holomorphes},
  Ann. Sci. \'{E}cole Norm. Sup. (4) \textbf{12} (1979), no.~2, 269--294.
  \MR{543218}

\bibitem{cattaneo-felder-groupoids}
A.~S. Cattaneo and G.~Felder, \emph{Poisson sigma models and symplectic
  groupoids}, Quantization of singular symplectic quotients, Progr. Math., vol.
  198, Birkh\"{a}user, Basel, 2001, pp.~61--93.

\bibitem{cieliebak-eliashberg}
K.~Cieliebak and Y.~Eliashberg, \emph{From {S}tein to {W}einstein and back},
  American Mathematical Society Colloquium Publications, vol.~59, American
  Mathematical Society, Providence, RI, 2012, Symplectic geometry of affine
  complex manifolds.

\bibitem{CDW}
A.~Coste, P.~Dazord, and A.~Weinstein, \emph{Groupo\"{\i}des symplectiques},
  Publications du {D}\'{e}partement de {M}ath\'{e}matiques. {N}ouvelle
  {S}\'{e}rie. {A}, {V}ol. 2, Publ. D\'{e}p. Math. Nouvelle S\'{e}r. A,
  vol.~87, Univ. Claude-Bernard, Lyon, 1987, pp.~i--ii, 1--62.

\bibitem{crainic-fernandes-lie}
M.~Crainic and R.~L. Fernandes, \emph{Integrability of {L}ie brackets}, Ann. of
  Math. (2) \textbf{157} (2003), no.~2, 575--620.

\bibitem{crainic-fernandes-poisson}
\bysame, \emph{Integrability of {P}oisson brackets}, J. Differential Geom.
  \textbf{66} (2004), no.~1, 71--137.

\bibitem{donaldson-monge}
S.~K. Donaldson, \emph{Holomorphic discs and the complex {M}onge-{A}mp\`ere
  equation}, J. Symplectic Geom. \textbf{1} (2002), no.~2, 171--196.

\bibitem{feix-cotangent}
B.~Feix, \emph{Hyperk\"{a}hler metrics on cotangent bundles}, J. Reine Angew.
  Math. \textbf{532} (2001), 33--46.

\bibitem{feix-twistor}
\bysame, \emph{Twistor spaces of hyperk\"{a}hler manifolds with
  {$S^1$}-actions}, Differential Geom. Appl. \textbf{19} (2003), no.~1, 15--28.

\bibitem{GW}
D.~Greb and M.~L. Wong, \emph{Canonical complex extensions of {K}\"{a}hler
  manifolds}, J. Lond. Math. Soc. (2) \textbf{101} (2020), no.~2, 786--827.

\bibitem{gualtieri-ham}
M.~Gualtieri, \emph{Generalized {K}\"{a}hler metrics from {H}amiltonian
  deformations}, Geometry and physics. {V}ol. {II}, Oxford Univ. Press, Oxford,
  2018, pp.~551--579.

\bibitem{hitchin-polygons}
N.~J. Hitchin, \emph{Polygons and gravitons}, Math. Proc. Cambridge Philos.
  Soc. \textbf{85} (1979), no.~3, 465--476.

\bibitem{hitchin-self}
\bysame, \emph{The self-duality equations on a {R}iemann surface}, Proc. London
  Math. Soc. (3) \textbf{55} (1987), no.~1, 59--126.

\bibitem{hitchin-def}
\bysame, \emph{Deformations of holomorphic {P}oisson manifolds}, Mosc. Math. J.
  \textbf{12} (2012), no.~3, 567--591, 669.

\bibitem{HKLR}
N.~J. Hitchin, A.~Karlhede, U.~Lindstr\"{o}m, and M.~Ro\v{c}ek,
  \emph{Hyper-{K}\"{a}hler metrics and supersymmetry}, Comm. Math. Phys.
  \textbf{108} (1987), no.~4, 535--589.

\bibitem{kaledin-2}
D.~Kaledin, \emph{A canonical hyperk\"{a}hler metric on the total space of a
  cotangent bundle}, Quaternionic structures in mathematics and physics
  ({R}ome, 1999), Univ. Studi Roma ``La Sapienza'', Rome, 1999, pp.~195--230.

\bibitem{kaledin-1}
\bysame, \emph{Hyperk\"ahler structures on total spaces of holomorphic
  cotangent bundles}, Hyperk\"{a}hler manifolds (Misha Verbitsky and Dmitri
  Kaledin, eds.), Mathematical Physics (Somerville), vol.~12, International
  Press, Somerville, MA, 1999, pp.~iv+257.

\bibitem{karasev-analogues}
M.~V. Karas\"{e}v, \emph{Analogues of objects of the theory of {L}ie groups for
  nonlinear {P}oisson brackets}, Izv. Akad. Nauk SSSR Ser. Mat. \textbf{50}
  (1986), no.~3, 508--538, 638.

\bibitem{karasev-quantization}
\bysame, \emph{The {M}aslov quantization conditions in higher cohomology and
  analogs of notions developed in {L}ie theory for canonical fibre bundles of
  symplectic manifolds. {I}, {II}}, vol.~8, 1989, Translated from the Russian
  by Pavel Buzytsky, Selected translations, pp.~213--234, 235--258.

\bibitem{kodaira}
K.~Kodaira, \emph{Complex manifolds and deformation of complex structures},
  Grundlehren der Mathematischen Wissenschaften [Fundamental Principles of
  Mathematical Sciences], vol. 283, Springer-Verlag, New York, 1986, Translated
  from the Japanese by Kazuo Akao, With an appendix by Daisuke Fujiwara.

\bibitem{kovalev-nahm}
A.~G. Kovalev, \emph{Nahm's equations and complex adjoint orbits}, Quart. J.
  Math. Oxford Ser. (2) \textbf{47} (1996), no.~185, 41--58.

\bibitem{kronheimer-cotangent}
P.~B. Kronheimer, \emph{{A hyperkahler structure on the cotangent bundle of a
  complex Lie group}}, MSRI Preprints (1988), arxiv:math/0409253.

\bibitem{kronheimer-ale}
\bysame, \emph{The construction of {ALE} spaces as hyper-{K}\"{a}hler
  quotients}, J. Differential Geom. \textbf{29} (1989), no.~3, 665--683.

\bibitem{kronheimer-coadjoint}
\bysame, \emph{A hyper-{K}\"{a}hlerian structure on coadjoint orbits of a
  semisimple complex group}, J. London Math. Soc. (2) \textbf{42} (1990),
  no.~2, 193--208.

\bibitem{kronheimer-nilpotent}
\bysame, \emph{Instantons and the geometry of the nilpotent variety}, J.
  Differential Geom. \textbf{32} (1990), no.~2, 473--490.

\bibitem{lgsx-holomorphic}
C.~Laurent-Gengoux, M.~Sti\'{e}non, and P.~Xu, \emph{Holomorphic {P}oisson
  manifolds and holomorphic {L}ie algebroids}, Int. Math. Res. Not. IMRN
  (2008), Art. ID rnn 088, 46.

\bibitem{lgsx}
\bysame, \emph{Integration of holomorphic {L}ie algebroids}, Math. Ann.
  \textbf{345} (2009), no.~4, 895--923.

\bibitem{maciocia}
A.~Maciocia, \emph{Metrics on the moduli spaces of instantons over {E}uclidean
  {$4$}-space}, Comm. Math. Phys. \textbf{135} (1991), no.~3, 467--482.

\bibitem{mackenzie-xu}
K.~C.~H. Mackenzie and P.~Xu, \emph{Integration of {L}ie bialgebroids},
  Topology \textbf{39} (2000), no.~3, 445--467.

\bibitem{morrey}
C.~B. Morrey, Jr., \emph{On the analyticity of the solutions of analytic
  non-linear elliptic systems of partial differential equations. {I}.
  {A}nalyticity in the interior}, Amer. J. Math. \textbf{80} (1958), 198--218.

\bibitem{nakajima-instantons}
H.~Nakajima, \emph{Instantons on {ALE} spaces, quiver varieties, and
  {K}ac-{M}oody algebras}, Duke Math. J. \textbf{76} (1994), no.~2, 365--416.

\bibitem{newlander-nirenberg}
A.~Newlander and L.~Nirenberg, \emph{Complex analytic coordinates in almost
  complex manifolds}, Ann. of Math. (2) \textbf{65} (1957), 391--404.

\bibitem{penrose}
R.~Penrose, \emph{Nonlinear gravitons and curved twistor theory}, Gen.
  Relativity Gravitation \textbf{7} (1976), no.~1, 31--52.

\bibitem{pym-constructions}
B.~Pym, \emph{Constructions and classifications of projective {P}oisson
  varieties}, Lett. Math. Phys. \textbf{108} (2018), no.~3, 573--632.

\bibitem{simpson-hodge}
C.~Simpson, \emph{The {H}odge filtration on nonabelian cohomology}, Algebraic
  geometry---{S}anta {C}ruz 1995, Proc. Sympos. Pure Math., vol.~62, Amer.
  Math. Soc., Providence, RI, 1997, pp.~217--281.

\bibitem{weinstein-local}
A.~Weinstein, \emph{The local structure of {P}oisson manifolds}, J.
  Differential Geom. \textbf{18} (1983), no.~3, 523--557.

\bibitem{weinstein-groupoids}
\bysame, \emph{Symplectic groupoids and {P}oisson manifolds}, Bull. Amer. Math.
  Soc. (N.S.) \textbf{16} (1987), no.~1, 101--104.

\bibitem{wells}
R.~O. Wells, Jr., \emph{Differential analysis on complex manifolds}, third ed.,
  Graduate Texts in Mathematics, vol.~65, Springer, New York, 2008, With a new
  appendix by Oscar Garcia-Prada.

\bibitem{yau}
S.~T. Yau, \emph{On the {R}icci curvature of a compact {K}\"{a}hler manifold
  and the complex {M}onge-{A}mp\`ere equation. {I}}, Comm. Pure Appl. Math.
  \textbf{31} (1978), no.~3, 339--411.

\bibitem{zakrzewski-1}
S.~Zakrzewski, \emph{Quantum and classical pseudogroups. {I}. {U}nion
  pseudogroups and their quantization}, Comm. Math. Phys. \textbf{134} (1990),
  no.~2, 347--370.

\bibitem{zakrzewski-2}
\bysame, \emph{Quantum and classical pseudogroups. {II}. {D}ifferential and
  symplectic pseudogroups}, Comm. Math. Phys. \textbf{134} (1990), no.~2,
  371--395.

\end{thebibliography}

\end{document}